\documentclass[a4paper,reqno,oneside,12pt]{amsart}
\usepackage{geometry}
\usepackage[utf8]{inputenc}
\usepackage{amssymb,amsfonts,amsmath,amsthm,eucal,url,cite}
\usepackage{enumitem}
\usepackage{tikz-cd}
\usepackage{hyperref}
\usepackage{multicol}
\usepackage{slashed}

\newcommand{\RR}{\mathbb{R}}
\newcommand{\CC}{\mathbb{C}}
\newcommand{\QQ}{\mathbb{Q}}

\newcommand{\ZZ}{\mathbb{Z}}

\newcommand{\F}{\mathcal{F}}

\newtheorem{theorem}{Theorem}[section]

\newtheorem{lemma}[theorem]{Lemma}

\newtheorem{conjecture}[theorem]{Conjecture}

\theoremstyle{definition}

\theoremstyle{remark}

\begin{document}
\title{Flat degenerate metrics and Riemannian foliations}

\author [B. Flamencourt]{Brice Flamencourt}
\address{UMPA, CNRS, \'Ecole Normale Sup\'erieure de Lyon, France}
\email{brice.flamencourt@ens-lyon.fr}

\maketitle

\begin{abstract}
Bandyopadhyay, Dacorogna, Matveev and Troyanov conjectured that a closed manifold admitting a flat, non-negative definite metric of constant rank $m$ should be finitely covered by a fiber bundle over the $m$-torus. We give a counter-example to this statement and we discuss the link between this problem and the study of transversely flat Riemannian foliations.
\end{abstract}

\section{Introduction}

An arbitrary $(0,2)$-tensor $h$ on a manifold $M$ is said to have flat coordinates if around any point there exist coordinates for which $h$ is a constant matrix. We will say in this case that $h$ is a flat bilinear form. In \cite{BDMT}, Bandyopadhyay, Dacorogna, Matveev and Troyanov studied the conditions under which $h$ is flat, generalizing the approach of Riemann in his introductory lecture of 1861, where he solved the case of a symmetric, positive definite $(0,2)$-tensor (see the two references cited in \cite{BDMT}, i.e. the original papers of Riemann \cite{Riem61, Riem56}). The authors also opened some questions about the global structure of manifolds admitting flat (degenerate) metrics, stating in particular the following conjecture:
\begin{conjecture} \label{mainconj}
Suppose a closed manifold M has a flat (possibly degenerate) non-negative definite metric g of rank m. Then, it is finitely covered by a manifold which is diffeomorphic to a fiber bundle over an $m$-dimensional torus.
\end{conjecture}
\noindent This conjectured was considered as a potential generalization of Bieberbach's famous theorem \cite{Bieb} about cocompact groups of isometries acting on $\RR^n$.

Our aim in this note is to provide a counter-example to this conjecture, by constructing a closed manifold with a flat metric of rank $3$, obtained as a suspension of the $4$-torus over $S^1$. This manifold is then covered by a fiber bundle over the $4$-torus, but it is not finitely covered by any fiber bundle over the $3$-torus.

Before proceeding with the construction of this example, we believe that it could be enlightening to establish some link between this conjecture and an intensively studied field of mathematics, the one of Riemannian foliations. The interested reader could find a very detailed introduction to this topic in the classical book of Molino \cite{Mol}.

Given a foliation $\F$ on a manifold $M^n$, we say that $\F$ is {\em transversely Riemannian} if there exists a fiber bundle metric on the normal bundle $T M / T \F \to $ of $\F$, which is projectable to the local quotient manifolds of the foliation. Now, we consider a manifold $M$ with a flat, non-negative definite (degenerate) metric $g$ of constant rank $m$. Around any point of $M$, there exist coordinates in which $g$ is a constant matrix, thus the isotropy cone $\mathcal C_g$ of $g$ does not depend on the point in these coordinates. In particular, $\mathcal C_g$ is an involutive distribution on $M$, which induces a foliation $\F$. In addition, the metric $g$ projects to a Riemannian metric on the local quotient manifolds of this foliation because $g$ is constant along the leaves of $\F$ (indeed, it is a constant matrix in the coordinates considered before). Consequently, $g$ defines a transverse Riemannian structure on $(M, \F)$ having the additional property that its induced transverse Levi-Civita connection is flat. It is obvious, conversely, that a flat transverse Riemannian structure on $(M, \F)$ induces a non-negative definite metric on $M$ which is flat in the sense of \cite{BDMT}.

Altogether, we can reformulate the conjecture~\ref{mainconj} in the following way: if $(M,\F)$ is a closed foliated manifold with a flat transverse Riemannian structure, then $M$ is finitely covered by a fiber bundle over the $q$-torus, where $q$ is the codimension of $\F$. However, such structures have already been studied by several authors. Indeed, they correspond to $(\mathrm{Isom}(\RR^q), \RR^q)$-transverse structures. A first interesting result in this direction is given by Yves Carrière in \cite[Theorem 4.2]{Car} (an english version can be found in \cite[Appendix A, Theorem 4.2]{Mol}), which states that in the case of an oriented one-dimensional foliation, $M$ is either diffeomorphic to $T^k \times P$ where $k > 1$ and $P$ is a flat manifold, or $M$ is a Seifert fibration. Nevertheless, this result does not contradict Conjecture~\ref{mainconj}. Other works on transverse $(G,T)$-structures have been carried out, studying more generally the transversely flat similarity structures (see \cite{Asu, Asu2} for example).

Our example, constructed in the next section, relies more on number theory in order to find a well-behaving diffeomorphism of the $4$-torus. To relate the construction to our previous discussion, we emphasize that this example is a compact manifold of dimension $5$ endowed with a foliation of codimension 2 admitting a transversely flat Riemannian structure.

\vspace{0.5cm}

{\bf Acknowledgments.} I would like to thank Vladimir S. Matveev for suggesting me to work on this problem. I also thank Abdelghani Zeghib for his helpful indications concerning references on transversely flat Riemannian foliations.

\section{Construction of the counter-example}

We define the matrix $A \in \mathrm{GL(\ZZ^4)}$ to be the companion matrix of the irreducible (over $\QQ$) polynomial $P(X) := X^4 - 2 X^3 - 2 X + 1$, i.e.
\begin{equation}
A := \left(\begin{matrix} 0 & 0 & 0 & -1 \\ 1 & 0 & 0 & 2 \\ 0& 1 & 0 & 0 \\ 0 & 0 & 1 & 2 \end{matrix} \right).
\end{equation}
In particular, $A$ is diagonalizable, and its eigenvalues are the roots of $P$. But one has
\begin{equation}
P(X) = (X^2 + (\sqrt{3} - 1) X + 1) (X^2 - (\sqrt{3} + 1) X + 1)
\end{equation}
and $(\sqrt{3} - 1)^2 - 4 < 0$ and $(\sqrt{3} + 1)^2 - 4 > 0$, so $P$ has two complex roots $\lambda$, $\bar\lambda$ of modulus $1$ and two real roots $\alpha$, $1/ \alpha$ different from $\pm 1$. Let $H$ be the plane of $\RR^4$ defined by $H := \ker (A^2 + (\sqrt{3} - 1) A + 1)$ and let $E := \ker(A^2 - (\sqrt{3} + 1) A + 1)$, so $\RR^4 = H \oplus E$.

We now consider the manifold $\tilde M := \RR^4 \times \RR$ and the group
\[
\Gamma := \ZZ^4 \rtimes \langle \RR^4 \times \RR \ni (x, t) \mapsto (Ax, t + 1) \rangle \simeq \ZZ^4 \rtimes \ZZ
\]
acting on $\tilde M$, where $\ZZ^4$ acts as the standard lattice on $\RR^4$. The restriction of $A$ to $H$ is diagonalizable in $\CC$ and has two distinct eigenvalues of modulus $1$, so it is an isometry for a positive definite quadratic form $q$ on $H$. We define the non-negative definite metric
\begin{equation}
\tilde g := (q \oplus 0_E) + dt^2
\end{equation}
on $\tilde M$, where we use the decomposition $\tilde M \simeq (H \oplus E) \times \RR$ and $t$ is the canonical coordinate of the last factor.

By construction, $\Gamma$ acts by isometries on $(\tilde M, \tilde g)$. Moreover, $\Gamma$ acts properly discontinuously, freely and cocompactly on $\tilde M$, so $M := \tilde M / \Gamma$ is a compact manifold and $\tilde g$ descends to a non-negative definite metric $g$ on $M$. This metric is flat because $\tilde g$ is flat and it has rank $3$.

Our goal is to prove that $M$ is not finitely covered by a fiber bundle over the $3$-torus $T^3$.

\begin{lemma} \label{subgroup}
A subgroup of $\Gamma$ is a semi-direct product of the form $L \rtimes K$ where $L$ is a subgroup of $\ZZ^4$ and $K$ is either the trivial group or the group generated by an element of $\Gamma$ which does not lie in $\ZZ^4$.
\end{lemma}
\begin{proof}
Let $\Omega$ be a subgroup of $\Gamma$. One has that $\Omega \cap \ZZ^4$ is a subgroup of $\ZZ^4$. Now, any element of $\Omega$ which does not lie in $\ZZ^4$ is of the form $n k$ where $n \in \ZZ^4$ and $k \in \{0\} \rtimes \ZZ \subset \Gamma$ is a non-trivial element. We consider
\begin{equation}
\mathcal E := \{k \in  \{0\} \rtimes \ZZ \subset \Gamma \ \vert \ \exists n \in \ZZ^4, \ n k  \in \Omega \}.
\end{equation}
It is obvious that $\mathcal E$ is a subgroup of $\ZZ$ in the decomposition $\Gamma = \ZZ^4 \rtimes \ZZ$, thus there exists $k_0 \in \mathcal E$ such that $\mathcal E = \langle k_0 \rangle$. Consequently, one can find $n_0 \in \ZZ^4$ such that $n_0 k_0 \in \Omega$ and any element of $\Omega$ can be written uniquely as $n (n_0 k_0)^m$ for elements $n \in \ZZ^4$ and $m \in \ZZ$.
\end{proof}

From Lemma~\ref{subgroup}, we deduce that any subgroup $\Omega$ of $\Gamma$ with finite index is of the form $L \rtimes K$ with $L$ a sublattice of $\ZZ^4$ and $K \simeq \ZZ$ is generated by a non-trivial element not lying in $\ZZ^4$. In particular, $\Omega \simeq \ZZ^4 \rtimes \ZZ$.

We assume by contradiction that $M$ is finitely covered by a fiber bundle over $T^3$ and we denote by $\bar M$ this finite cover. Since $\pi_1(\bar M)$ is a subgroup of $\Gamma$ with finite index, one has $\pi_1(\bar M) \simeq L \rtimes K$ with $K \simeq \ZZ$ using the notations of the previous discussion. We write the long sequence of homotopy groups of $F \to \bar M \to T^3$ (where $F$ is the typical fiber):
\begin{equation}
\pi_2(T^3) \to \pi_1(F) \to \pi_1(\bar M) \to \pi_1(T^3) \to \pi_0(F),
\end{equation}
which becomes a short exact sequence:
\begin{equation}
0 \to \pi_1(F) \to L \rtimes K \to \ZZ^3 \to 0.
\end{equation}
This means that $\pi_1(F)$ is a normal subgroup of $L \rtimes K$. Since $(L \rtimes K) / \pi_1(F) \simeq \ZZ^3$ is abelian, $\pi_1(F)$ should contain the commutator subgroup of $\ZZ^4 \rtimes \ZZ$. If we denote by $k_0$ a generator of the group $K$, the commutator subgroup contains all the elements of the form $n^{-1} k_0^{-1} n k_0$ for $n \in L$. But $k_0$ is a map of the form $\RR^4 \times \RR \ni (x, t) \mapsto (A^m x + \tau, t + m)$ for some $m \in \ZZ \setminus \{0\}$ and $\tau \in \ZZ^4$. We deduce that $n^{-1} k_0^{-1} n k_0 (x, t )= (x + (A^m n) - n, t) = (x + (A^m - I_4) n, t)$. By construction, $A^m$ has all its eigenvalues different from $1$, so $A^m - I_4$ is non-singular, implying that $(A^m - I_4) L$ is a sublattice of $\ZZ^4$ (in particular, it is an abelian group of rank $4$). The rank of the abelian group $(\ZZ^4 \rtimes \ZZ) / \pi_1(F) \simeq \ZZ^3$ should thus be less than $1$, which is a contradiction.


\begin{thebibliography}{10}

\bibitem{Asu}
T. Asuke, {\sl On transversely flat conformal foliations with good measures}. Transaction of the American Mathematical Society, {\bf 348} (5), pp. 1939--1958 (1996).

\bibitem{Asu2}
T. Asuke, {\sl Classification of riemannian flows with transverse similarity structures}. Annales de la faculté des sciences de Toulouse, Série 6, {\bf 6} (2), pp. 203--227(1997).

\bibitem{BDMT}
S. Bandyopadhyay, B. Dacorogna, V. S. Matveev, M. Troyanov, {\sl Bernhard Riemann 1861 revisited: existence of flat
coordinates for an arbitrary bilinear form}. Mathematische Zeitschrift,  {\bf 305} (12) (2023).

\bibitem{Bieb}
L. Bieberbach, {\sl Über die Bewegungsgruppen der Euklidischen Räume}. Math. Ann., {\bf 70} (3), 297–336 (1911).

\bibitem{Car}
Y. Carrière, {\sl Variations sur les flots riemanniens}. Séminaire de théorie spectrale et géométrie, {\bf 5}, pp. 43-55 (1986-1987).

\bibitem{Mol}
P. Molino, {\it Riemannian foliations}. Progress in Mathematics. {\bf 73}  Birkhäuser Boston, Inc., Boston, MA (1988).

\bibitem{Riem61}
B. Riemann, {\sl Commentatio mathematica, qua respondere tentatur quaestioni ab Ill ma Academia Parisiensi propositae: “Trouver quel doit être l’ état calorifique d’un corps solide homogène indéfeni pour qu’un système de courbes isothermes, à un instant donné, restent isothermes après un temps quelconque, de telle sorte que la température d’un point puisse s’exprimer en fonction du temps et de deux autres variables indèpendantes.”} Gesammelte Mathematische Werke, Erste Auflage, pp. 370–384, Teubner 1876/ Zweite Auflage, Teubner, pp. 391–404 (1892).

\bibitem{Riem56}
B. Riemann, {\sl Ueber die Hypothesen, welche der Geometrie zu Grunde liegen. Abhandlungen der Königlichen Gesellschaft der Wissenschaften zu Göttingen}, {\bf 13} (1867) and Gesammelte Mathematische Werke, Erste Auflage, pp. 370–384, Teubner, 254—270 (1876).

\end{thebibliography}
\end{document}